\documentclass[a4paper,12pt]{article}
\usepackage{ucs} 
\usepackage[utf8x]{inputenc}
\usepackage{amsfonts} % For symbols for numbersets
\usepackage{amsmath} % For templates like underscript
\usepackage[pdftex]{graphicx} %for graphics
\usepackage{caption} % for caption on graphics
\usepackage{subcaption} % for sub figures and subcaptions
\usepackage{sidecap} % for side captions
\usepackage{enumerate}
\usepackage{setspace}
\usepackage[toc,page]{appendix}
\usepackage{color}

\title{}
\author{}
\date{}

\newtheorem{theorem}{Theorem}[section]
\newtheorem{lemma}[theorem]{Lemma}
\newtheorem{proposition}[theorem]{Proposition}
\newtheorem{corollary}[theorem]{Corollary}

\newenvironment{proof}[1][Proof]{\begin{trivlist}
\item[\hskip \labelsep {\bfseries #1}]}{\end{trivlist}}

\newenvironment{keywords}{\begin{@abssec}{\keywordsname}}{\end{@abssec}}
\newenvironment{@abssec}[1]{%
\if@twocolumn
\section*{#1}%
\else
\vspace{.05in}\footnotesize
\parindent .2in
{\upshape\bfseries #1. }\ignorespaces 
\fi}
{\if@twocolumn\else\par\vspace{.1in}\fi}
\newenvironment{AMS}{\begin{@abssec}{\AMSname}}{\end{@abssec}}

\makeatletter
\def\imod#1{\allowbreak\mkern10mu({\operator@font mod}\,\,#1)}
\makeatother
\newcommand{\qed}{\nobreak \ifvmode \relax \else
\ifdim\lastskip<1.5em \hskip-\lastskip
\hskip1.5em plus0em minus0.5em \fi \nobreak
\vrule height0.75em width0.5em depth0.25em\fi}
\newcommand{\Torus}{\mathbb{T}^2} 
\newcommand{\TorusD}{\mathbb{T}^d} 
\newcommand{\TorusP}[1]{\mathbb{T}^{#1}}
\newcommand{\EucP}[1]{\mathbb{R}^{#1}} 
\newcommand{\EucD}{\mathbb{R}^d} 

\newcommand{\Cone}{\mathcal{C}}
\newcommand{\upPhi}{\hat{\Phi}}
\newcommand{\lowPhi}{\Phi}
\newcommand{\vp}{v_m^L}
\newcommand{\proj}{\mbox{proj}}
\newcommand{\Tile}{P}
\newcommand{\Plane}{\mathcal{P}}

\newcommand\keywordsname{Key words}
\newcommand\AMSname{AMS subject classifications}
\begin{document}
\title{Indestructible dynamics of torus maps}
\author{Suddhasattwa Das\footnotemark[1], \and James A Yorke\footnotemark[3]}
\footnotetext[1]{Department of Mathematics, University of Maryland, College Park}
\footnotetext[3]{University of Maryland, College Park}
\date{\today}
\maketitle

\begin{abstract}
Given a $d$-dimensional torus map  $F(z)=Mz+G(z)\bmod 1$, where $M$ is an integer-matrix and and $G$ is a periodic function, we find conditions on $M$ under which $F$ is semi-conjugate to a linear torus map, independently of $G$. We also find a conditions $G$ under which these semi-conjugacies can be turned into conjugacies. These conditions are satisfied by open sets of torus maps (in the $C^1$-topology) and therefore describe some asymptotic behavior of trajectories which are stable under perturbations to the map.
\end{abstract}



%-_-_-_-_-_-_-_-_-_-_-_-_-_-_-_-_-_-_-_-_-_-_-_-_-_-_-_-_-_-_-_-_-_-_-_-_-_-_-_-_-_-_-_-_-_-_-_-_-_-_-__-_-_-_-_-_-_-_-_-_-_-_-_-_-_-_-_-_-_-_-_-_-_-_-_-_-_-_-_-_-_-_-_-_
\section{Introduction}

Recall that the torus $\TorusD$ is the Cartesian product of $d$ topological circles $S^1\times\ldots\times S^1$, where $S^1$ is represented as the interval $[0,1]$ with the endpoints identified. Any continuous torus map $F:\TorusD\to\TorusD$ is of the form shown below.
\begin{equation}\label{eqn:periodic_part}
F(z \bmod 1)=Mz+G(z) \bmod{1},
\end{equation}
where $M$ is a $d\times d$ integer matrix (entries are integer-valued) that we call the \textbf{winding matrix} of $F$, and $G:\EucD\to\EucD$ is a bounded, 1-periodic function, called the \textbf{periodic part} of $F$.
%[A map $h:\mathbb{R}^2\rightarrow\mathbb{R}$ is said to be \textbf{1-periodic} in $x$ and $y$ if for every pair of integers $n_1$ and $n_2$, $h(x+n_1,y+n_2)=h(x,y)$]. 
Note that if for all $z\in\TorusD, |\det M(z)|=m>0$ and if $|\det dF|>0$, then $F$ is an $m$-to-$1$ of the torus, so every point has exactly $m$ pre-images under $F$. Any continuous torus map can be separated into a linear and periodic part as shown above, and is shown in \cite{Conjug2}, these are unique. 

\textbf{Note.} The winding matrix $M$ is also the map $F_*:H_1(\TorusD)\to H_1(\TorusD)$ on the first-homology group.

The following theorem gives some condition under which a torus map $F$ is semi-conjugate to a linear map on a lower dimensional torus, namely, $x\mapsto Ax \bmod 1$. \boldmath $SL(d,\mathbb{Z})$ \unboldmath denotes the set of $d\times d$ integer matrices with determinant $\pm 1$ (i.e., $SL(d,\mathbb{Z})$ ). We will use the following assumption.

\bigskip

\noindent \textbf{(A1)} There is some $S\in SL(d,\mathbb{Z})$ such that $S^{-1}MS$ is in block upper-triangular form, where $M$ is defined in Eq. \ref{eqn:periodic_part}. Then the top left $k\times k$ block will be denoted as \boldmath $A$ \unboldmath(where $k\leq d$). Let \boldmath $W$ \unboldmath be the subspace spanned by \{$e_1\ldots,e_k$\}, where \boldmath $e_i$ \unboldmath is the vector with all entries $0$ other than the $i$-th entry, which is $1$. 

\begin{theorem}[A semi-conjugacy theorem]\label{thm:factor_Phi_matrix}
Let $F:\TorusD\to\TorusD$ be a $d$-dimensional continuous torus map of the form (\ref{eqn:periodic_part}). Assume (A1). If either of the two conditions hold,
\\(i) $F$ is invertible and $A$ is a hyperbolic matrix; or
\\(ii) $A$ is an expanding matrix.
\\then $F$ is semi-conjugate to the following map on $\TorusP{k}$.
\begin{equation}\label{eqn:expnd_torus}
\theta_{n+1}=A\theta_n\mod 1.
\end{equation}
In other words, there is an onto map $\lowPhi:\TorusD\to\mathbb{T}^k$ such that $\lowPhi F(z))=A\lowPhi(z) \bmod 1$. Moreover, in case (ii) is satisfied, then for every $\theta\in \TorusP{k}$, $\lowPhi^{-1}(\theta)$ intersects every $k$-dimensional sub-torus of $\TorusD$ parallel to $\mathcal{F}$. In particular, $\lowPhi$ is onto.
\end{theorem}
Theorem \ref{thm:factor_Phi_matrix} will be proved in Section \ref{subsect:Phi_matrix}. 

\textbf{Remark - the 1-D case.} Consider the case $d=1$ and $F:S^1\to S^1$ is a degree-$m$ circle map with an attracting fixed point $P$. Then $F$ is of the form $F(z)=mz+G(z)\bmod 1$ for some $m>1$, and by Theorem \ref{thm:factor_Phi_matrix}, it is semi-conjugate to the expanding circle map below, which has no attractor. The semi-conjugacy is given by $\lowPhi(z\bmod 1)=\underset{n\to\infty}{\lim}m^{-n}F^n(z) \bmod 1$.
\begin{equation}\label{eqn:expnd_circ}
x_{n+1}=mx_n\mod 1.
\end{equation}
The basin of attraction of $P$ is a countable union of open intervals and each such interval must be mapped by $\lowPhi$ into points that eventually map into $\lowPhi(P)$ by the map in Eq. \ref{eqn:expnd_circ}. If the basin of $P$ is dense, then its complement is a Cantor set which is mapped onto $S^1$ by $\lowPhi$. Hence, this semi-conjugacy can lose some information about the dynamics while preserving other information. 

\textbf{Remark.} Note that if $k=d=1$ and $F$ is an expanding map, then the semi-conjugacy is a conjugacy, a fact that also follows from a theorem by Shub \cite{ExpndngEndo}.

\textbf{Remark.} The assumption on the matrix $M$ being conjugate to an integer-valued block upper-triangular matrix cannot be replaced by the assumption that there are eigenvalues not equal to one in magnitude. Any monic, integer polynomial is the characteristic polynomial of some integer-valued matrix and the polynomial be  irreducible over the ring of integers but have eigenvalues not equal to one in magnitude. An example is the Lehmer polynomial \cite{Lehmer_polyn} $x^10+x^9-x^7-x^6-x^5-x^4-x^3+x+1$. It is known to be irreducible but having two roots different from one in magnitude.

\begin{corollary}[Semiconjugacy to an expanding circle map]\label{corr:factor_Phi}
Let $F$ be  a $d$-dimensional continuous torus map whose winding matrix $M$ has an integer eigenvalue $m$. Then $F$ is semi-conjugate to the map in Eq. \ref{eqn:expnd_circ}, i.e., there is a map $\lowPhi:\TorusD\to\TorusD$ such that $\lowPhi F(z))=m\lowPhi(z) \bmod 1$.
\end{corollary}
\begin{proof} If the winding matrix $M$ has a left eigenvector $v^L_m$ with integer eigen value $m>1$, then this vector can be used to construct a unit volume tiling of $\EucD$ which we describe below. This tiling can be used to contruct a matrix $S\in SL(d,\mathbb{Z})$ from Theorem \ref{thm:factor_Phi_matrix} such that $k=1$, and $A\equiv m$, and Corollary \ref{corr:factor_Phi} is proved. \qed
\end{proof}

\textbf{A tiling of the torus.} Given the integer vector $\vp \in\mathbb{R}^d$, we can find $d$ linearly independent integer vectors $w_1,\ldots,w_{d}$ which satisfy the following.
\\(i) For each $i=1\ldots,d-1$, $w_i$ is perpendicular to $\vp$.
\\(ii) The volume of the parallelotope \boldmath $\Tile$ \unboldmath formed by \{$w_1,\ldots,w_{d}$\} have unit volume.
\\(iii) $\phi(z)=0$ on the side of $\Tile$ which contains $w_1,\ldots,w_{d-1}$, and is $1$ on the opposite face of $\Tile$.

To see this, first let $\Plane$ be the $(d-1)$-hyperplane of vectors perpendicular to $\vp$. Since $v^L_m$ is integer vector, there are $d-1$ linearly independent integer vectors perpendicular to $v^L_m$. So we can pick $d-1$ linearly independent integer vectors $w_1,\ldots,w_{d-1}$ in $\Plane$ such that the $(d-1)$-dimensional parallelotope described by $w_1,\ldots,w_{d-1}$ has no lattice point in its interior or in the interior of any of its faces. There are two lattice points away from $\Plane$ which are closest to $\Tile$. Let $w_d$ be the one among them for which $v^L_m\cdot w_d>0$. Let $\Tile$ be the parallelotope described by \{$w_1,\ldots,w_{d}$\}. 

Now note that the $\Tile$ has no lattice point in its interior or in the interior of any of its faces. Therefore, by the $d$-dimensional Pick's formula, $\Tile$ has volume 1. Therefore, $\Tile$ forms a periodic tiling of $\mathbb{R}^d$ and $\Tile\mod Z\cong \TorusD$.

See Fig. \ref{fig:plane_tiling} for an example when $d=2$, $k=1$.

\begin{figure} %XXX
\includegraphics[scale=0.2]{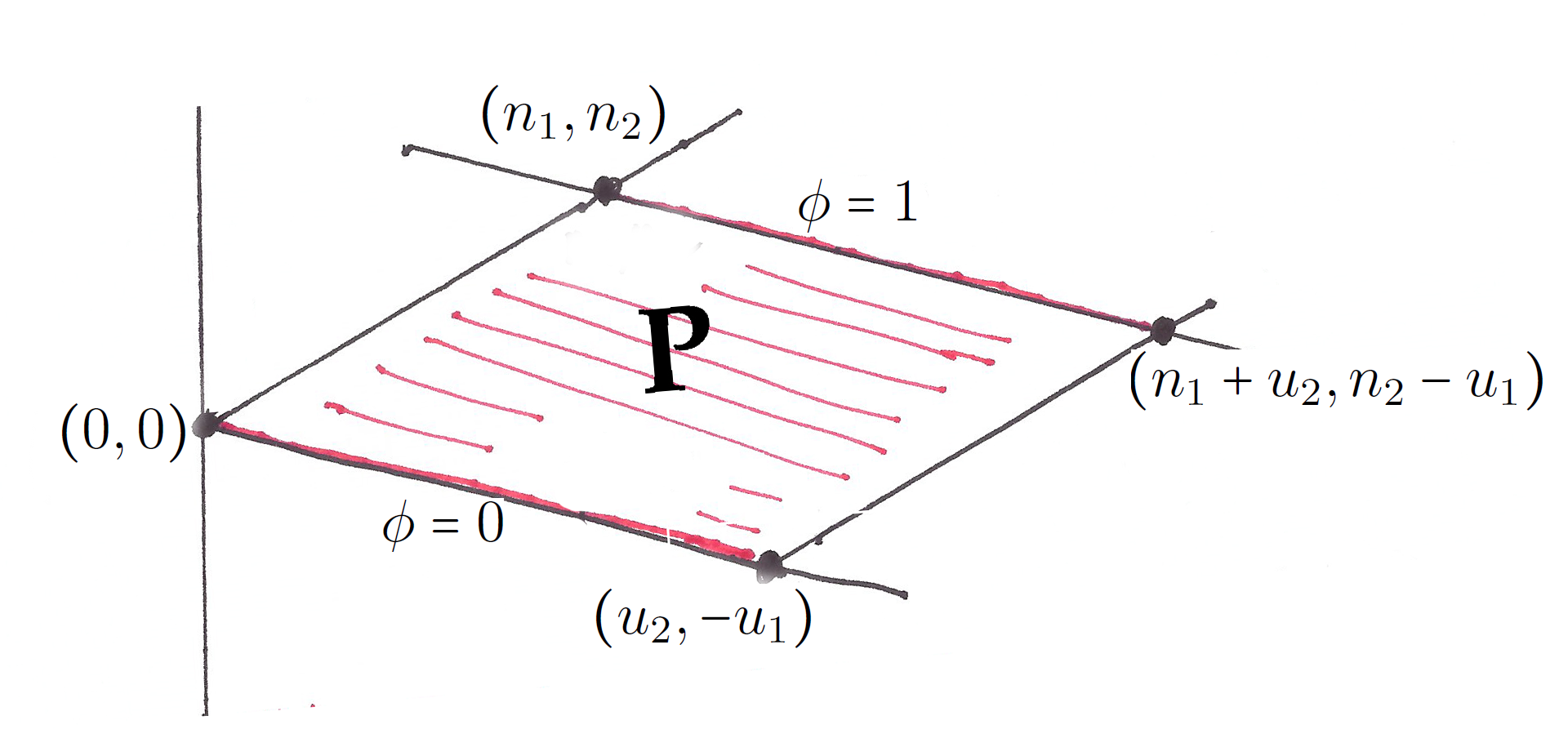}
\caption{\textbf{A lift $P$ of $\Torus$.} Let 
$(u_1, u_2) := v_m^L$ be the left eigenvector of the winding matrix $M$ for which the integers $u_1$ and $u_2$ are chosen as small as possible; 
that is, they are relatively prime; that is, we can choose integers $n_1$ and $n_2$ for which $n_1 u_1 + n_2 u_2 = 1$. 
The parallelogram $P$ shown has area $1$. 
The map $(x,y)\mapsto (x\mod1,y\mod1)$ takes the parallelogram $P$ onto the torus in a one-to-one manner -- except on the boundary of the parallelogram.
The vector $v_1^R:=(u_2,-u_1)$ is a right eigenvector of $M$ and 
has eigenvalue $1$.
Define that $\phi(z):=\vp\cdot z$. 
Then $\phi$ is $0$ on the line containing $(0,0)$ and $(u_2,-u_1)$ and is $1$ on the parallel line containing $(n_1,n_2)$, and $\phi$ is constant on every line parallel to these.
}
\label{fig:plane_tiling}
\end{figure}

\textbf{Skew-product maps.} An important class of maps which have many interesting topological and measure-theoretic properties and which can be factored as in Corollary \ref{corr:factor_Phi} are ``skew product'' maps. We are interested in skew-product maps of the form
\begin{equation}\label{eqn:map}
F(x,y) = (Ax,\ F_y(x,y))\ \mbox{mod 1 (in each coordinate)}.
\end{equation}
where $x\in \TorusP{k}$, $y\in \TorusP{d-k}$, $A$ is a $k\times k$ expanding, integer-matrix, and $F_y:\TorusP{d-k}\to\TorusP{d-k}$ is continuous. Note that the map in Eq. \ref{eqn:map} is semi-conjugate to the linear torus map in Eq. \ref{eqn:expnd_torus}.

See \cite{SkewProd5} for a nice overview of skew-product maps and their treatment as random maps or non-autonomous differential equations. In \cite{SkewProd2}, Kleptsyn and Nalskii, using the viewpoint of stochastic circle diffeomorphisms, prove the important result that under certain conditions, the orbits of almost every (with respect to a certain measure) pair of initial conditions on a fiber are asymptotic. Homburg \cite{SkewProd1} looks at skew products with expanding circle maps and prove the occurrence of topological mixing for an open set of maps. The authors of \cite{Kostelich} studied the mechanism by which chaos occurs on a certain class of skew-product maps. Ilyashenko and Negut \cite{InvsblAttrctr} constructed a family of structurally stable skew product maps over the Smale-Williams horseshoe in which the attract has arbitrarily low information dimension. Other features, like the measure of the non-wandering set, and perturbation of skew product systems, has been in investigated in \cite{SkewProd4} and \cite{SkewProd3} respectively. %Our main theorem describes a large family of torus maps which are topologically conjugate to maps of the form (\ref{eqn:map}) and hence, have many of the properties of skew-product maps.

Recall that two dynamical systems are said to be \textbf{conjugate} if there is a change of coordinates that is continuous and with a continuous inverse, transforming one dynamical system to the other. This change of variables is called a \textbf{conjugacy}.  We will state and prove two conjugacy results, Theorem \ref{theo:conj_exp_cone} and Corollary \ref{corr:conj_exp_cone}, later in Section \ref{sect:conjg}, which provide sufficient conditions under which a torus map is conjugate to a skew-product map. 

%-_-_-_-_-_-_-_-_-_-_-_-_-_-_-_-_-_-_-_-_-_-_-_-_-_-_-_-_-_-_-_-_-_-_-_-_-_-_-_-_-_-_-_-_-_-_-_-_-_-_-__-_-_-_-_-_-_-_-_-_-_-_-_-_-_-_-_-_-_-_-_-_-_-_-_-_-_-_-_-_-_-_-_-_
\section{Proof of Theorem \ref{thm:factor_Phi_matrix}}\label{subsect:Phi_matrix}

By assumption, there is a matrix $S\in SL(d,\mathbb{Z})$ such that $SMS^{-1}$ is in block upper-triangular form. Such a matrix $S$ corresponds to a change of coordinates in $\TorusD$, so we will assume that we are working in these coordinates and $M$ is in the desired form. Then note that $W$ is an invariant subspace of $M$ and $M|W=A$. Let \boldmath $\mathcal{F}$ \unboldmath denote the face of the torus spanned by \{$e_1\ldots,e_k$\}. We will first prove the case when $A$ is hyperbolic and $F$ invertible, and make some minor modifications to the proof to prove the case when $A$ is expanding and $F$ not necessarily invertible.

%...............................................................................................................................................................................................................................................................
\subsection{The case when $A$ is hyperbolic.}

Since $A$ is hyperbolic, the subspace $W$ splits into two complementary, invariant subspaces $W=W^u\oplus W^s$ such that $A_u:=A|W^u$ and $A_s:=A|W^s$ are expanding and contracting respectively. Let the dimension of these spaces be $k^u$ and $k^s$ respectively. then note that $k=k_u+k_s$. $A_u$ and $A_s$ can be viewed as square matrices of dimension $k_u, k_s$ respectively.Define
\\$\upPhi_u:\EucD\to W^u$; $\upPhi_u(z):=\underset{n\to\infty}{\lim}A_u^{-n}\proj_{W^u}F^{n}(z)$.
\\$\upPhi_u:\EucD\to W^s$; $\upPhi_s(z):=\underset{n\to\infty}{\lim}A_s^{n}\proj_{W^s}F^{-n}(z)$.

\textbf{Claim A.} $\upPhi_u$ and $\upPhi_s$ are well defined and continuous. We will only prove the claim for $\upPhi_u$, as the proof for $\upPhi_s$ is analogous. To prove that the limit exists, we will use the following equation which follows from the definition of $\upPhi_u$, in which we express $\upPhi_u$ as an infinite series.
\begin{equation}\label{eqn:Phi_series}
\upPhi_u(z)= \proj_{W^u} z+\underset{k=1,2,\ldots,}{\Sigma}A_u^{-k}\proj_{W^u} G\circ F^{k-1}(z)
\end{equation}
Since $G$ is $1$-periodic, it is uniformly bounded. Since $\|A_u^{-1}\|<1$ and each of the terms $\proj_{W^u} G\circ F^{k-1}(z)$ are bounded, the limit $\upPhi_u(z)$ exists as an uniform limit, by the Weierstrass M-test (see \cite{Rudin1}). Since each finite sum is a continuous function, by the Uniform Limit Theorem (see \cite{Rudin1}), $\upPhi_u(z)$ is also continuous.

\textbf{Claim B.} $\upPhi_u$ is a semi-conjugacy. Note that $\upPhi_u\circ \hat{F}(z)=\underset{n\rightarrow\infty}{\lim}A_u^{-n}\proj_{W^u}F^{n+1}(z) = A_u\underset{n\rightarrow\infty}{\lim}A_u^{-n-1}\proj_{W^u}F^{n+1}(z)= A_u\upPhi_u(z)$. A similar result holds for $\upPhi_s$. Therefore, we have proved the following.
\begin{equation}\label{eqn:factoring_phi}
\mbox{For every } z\in\EucD, \upPhi_u\circ\hat{F}(z)=A_u \upPhi_u(z),\ \ \upPhi_s\circ\hat{F}(z)=A_s \upPhi_u(z).
\end{equation}

Let \boldmath $\upPhi$ \unboldmath $:\EucD\to \mathbb{R}^k$ be the map $\upPhi(z)=\upPhi_u(z)\oplus \upPhi_s(z)$. Then note that
\[\upPhi(F(z))=A_u \upPhi_u(z)+A_s \upPhi_s(z)=A\upPhi(z).\]

Secondly, let $\vec{m}\in\mathbb{Z}^d$. Then : $F(z+\vec{m})=M(z+\vec{m})+G(z+\vec{m})=F(z)+M\vec{m}$. Applying this formula $n$ times gives : $F^{n}(z+\vec{m})=F^{n}(z)+M^n\vec{m}.$

Similar to Eq. \ref{eqn:Phi_series}, we can write $\upPhi(z)$ as the infinite sum
\[\upPhi(z)=\proj_W z + \underset{n=1,2,\ldots}{\Sigma}[A_u^{-n}\proj_{W^u}G(F^{n-1}z)\oplus A_s^{n}\proj_{W^s}G(F^{n-1}z)]\]
\begin{equation}\begin{split}\label{eqn:up_phi_periodic}
\upPhi(z+\vec{m}) &=\proj_W (z+\vec{m}) + \underset{n=1,2,\ldots}{\Sigma}[A_u^{-n}\proj_{W^u}G(F^{n-1}(z+\vec{m}))\oplus A_s^{n}\proj_{W^s}G(F^{n-1}(z+\vec{m}))]\\
&= \proj_W z + \underset{n=1,2,\ldots}{\Sigma}[A_u^{-n}\proj_{W^u}G(F^{n-1}(z+\vec{m}))\oplus A_s^{n}\proj_{W^s}G(F^{n-1}(z+\vec{m}))] + \proj_W \vec{m}\\
&=\upPhi(z)+\proj_W \vec{m}
\end{split}\end{equation}
Note that $\proj_W \vec{m}\in\mathbb{Z}^k$. Therefore the map $\lowPhi:\TorusD\to\mathbb{T}^k$ defined as
\[\lowPhi(z\bmod 1):=\upPhi(z)\bmod 1\]
is well defined. The map $\lowPhi$ is the desired semi-conjugacy map. This completes the proof of the of Theorem \ref{thm:factor_Phi_matrix} for the hyperbolic case. \qed

%...............................................................................................................................................................................................................................................................
\subsection{The case when $A$ is expanding.}

Let $\proj_W$ be the orthogonal projection onto the subspace $W$. Then
\[\upPhi(z):=\underset{n\to\infty}{\lim} A^{-n} \proj_W F^{n}(z)\]
In a manner similar to the hyperbolic case, it can be shown that $\upPhi$ is well defined, continuous and factors into a map $\lowPhi:\TorusD\to\mathbb{T}^k$ which serves as the semi-conjugacy. 

To prove that $\lowPhi$ is onto, it is equivalent to prove the analogous statement for $\upPhi$, namely

\textbf{Claim C.} For every $x\in \EucP{k}$, $\upPhi^{-1}(x)$ intersects every $k$-sub-plane of $\TorusD$ parallel to $W$. In particular, $\upPhi$ is onto.

The following inequalities follows from the definition of $\upPhi$ and will be important for making conclusions about the fibers of $\upPhi$. Let $\|G\|_0$ denote the $C^0$-norm of $G$, defined as $\underset{z\in\EucD}{\sup}\|G(z)\|$. 
\begin{equation}\label{eqn:Phi-phi_bound}
\mbox{For every } z\in\EucD,\ \ |\upPhi(z)-\proj_W(z)|\leq\frac{1}{\|A\|-1}\|\|\|G\|_0
\end{equation}
\begin{equation}\label{eqn:ReverseLip_Phi}
\mbox{For every } z_1, z_2\in\EucD,\ \left||\upPhi(z_1)-\upPhi(z_2)|-|\proj_W(z_1)-\proj_W(z_2)|\right|\leq \frac{2}{\|A\|-1}\|G\|_0
\end{equation}

\textbf{Proof of Claim C.} Let the contrary to this statement be true. Let $S$ be a $k$-sub-plane of $\TorusD$ parallel to $W$. If $\upPhi^{-1}(x)$ does not intersect $S$, it means that the image $\upPhi(S)$ is bounded in some direction. However, the image $\proj_W(S)$ is $\EucP{k}$. This leads to a contradiction of Ineq. \ref{eqn:Phi-phi_bound}. \qed

%-_-_-_-_-_-_-_-_-_-_-_-_-_-_-_-_-_-_-_-_-_-_-_-_-_-_-_-_-_-_-_-_-_-_-_-_-_-_-_-_-_-_-_-_-_-_-_-_-_-_-__-_-_-_-_-_-_-_-_-_-_-_-_-_-_-_-_-_-_-_-_-_-_-_-_-_-_-_-_-_-_-_-_-_
\section{When torus maps are conjugate to skew-products.}\label{sect:conjg}

%...............................................................................................................................................................................................................................................................
\subsection{Cone structures.}\label{subsect:cones}

%We will first introduce a notion called ``invariant, expanding cones''.

\textbf{Invariant, expanding cones.} %Much of what we do can be immediately extended to higher dimensions, but here we stay with $\Torus$. 
The property of ``dominating expansion'' can be defined on any manifold, but we will stick to $\TorusD$. Let $e(z)$ be a smooth $k$-sub-bundle on $\TorusD$, i.e., $e(z)$ is a $k$-dimensional vector sub-space of the tangent space at the poiont $z$. Let $e^\perp(z)$ be the orthogonal complement of $e(z)$. A tangent vector $v$ at a point $z\in\TorusD$ can be uniquely represented as $v=(a,b)_e$, where $a\in e(z)$, $b\in e^\perp(z)$. Let $v'=dF(z)v$ and let the representation of $v'$ in terms of the vector spaces $e(F(z)),e^\perp(F(z))$ be $(a',b')_e$. We say that there an \textbf{expanding cone structure} centered around $e$ if there are constants $K>1$ and $\alpha>0$ such that for every point $z$, if $|b|\leq\alpha|a|$, then,\\
\noindent
(i) $|b'|\leq\alpha|a'|$, and\\
(ii) $|a'|\geq K|a|$.\\
We can rephrase that as follows. At every point $z\in\TorusD$, the \boldmath $\alpha$-\textbf{cone}\unboldmath, denoted $\Cone_\alpha(z)$ is the set of vectors $v$ in the tangent space at $z$ such that $\|\proj_{e^\perp}v\|\leq\alpha\|\proj_{e}v\|$. This cone-structure is said to be invariant, expanding under $F$ if if for some $\alpha>0$ and some $K >1$, $\Cone_\alpha(F(z))\subset DF(z)(\Cone_\alpha(z))$ and if $(a',b')_e=DF(z)(a,b)_e$, then $|a'|>K||a|$. See Fig. \ref{fig:inv_cone} for a schematic diagram.

The cone structure is said to be \textbf{dominated} if for each non-zero vector $v\in e^\perp(z)$, $\|DF(z)v\|<K\|v\|$. In other words, the expansion is the strongest for vectors within the cone.

Invariant cone systems are in particular, present in hyperbolic systems, as also in various weaker forms of hyperbolicity like dominated cones \cite{SeminalDominated} and dominated splittings \cite{DomSplit}. Note that in non-hyperbolic systems, the tangent subspaces $e$ and $e^\perp$ are not invariant. Most of techniques used to prove properties in these different versions of hyperbolicity cannot be extended to the broader class of maps we are interested in, like maps which are either not diffeomorphisms or without a continuous invariant splitting of the tangent space.

\begin{figure}
\includegraphics[width=0.5\textwidth]{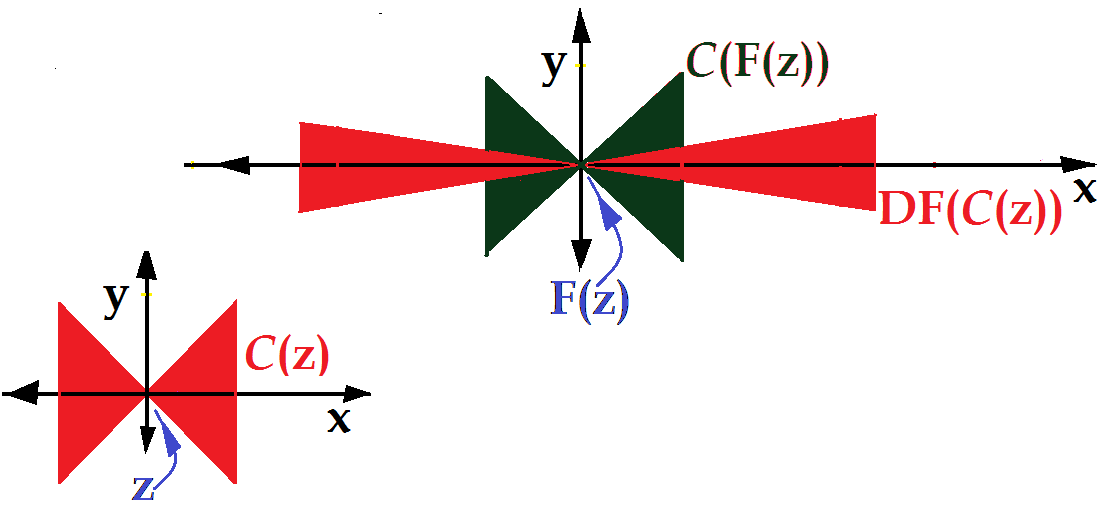}
\caption[Invariant, expanding cones.]{\textbf{Invariant, expanding cones.} This figure illustrates an invariant, expanding cone structure for a torus map $F:\Torus\to\Torus$. For the sake of simplicity, $e^1$ and $e^2$ have been taken to be the unit vector fields along the X and Y directions of the torus respectively, and $\alpha=1$. The triangles drawn in red and green lie in the tangent spaces at $z$ and $F(z)$ respectively, and pictorially represent part of the cones containing the vectors $\{(u,v)\ |\ \|v\|\leq \|u\|\leq 1\}$ in their respective spaces. Any vector within the red cone $\Cone(z)$ at $z$ is mapped into the green cone $\Cone(F(z))$ at $F(z)$ under the action of $DF(z)$ and also stretched by a factor of at least $K>1$.}
\label{fig:inv_cone}
\end{figure}

\textbf{Definition [conic curve].} A differentiable curve in $\EucD$ or in $\TorusD$ is said to be a {\bf conic curve} if its tangent vector at every point lies inside the expanding cone at that point on the manifold. Note that the image of a conic-curve under the map is again a conic curve, with an expansion in length by a factor of at least $K$. Also, since $W$ lies inside the expanding cone by assumption (A2), there is a uniform constant $\tau>0$ such that for any conic curve of length $l$ joining two points $A$ and $B$, 
\begin{equation}\label{eqn:length_proj}
|\proj_W (A-B)|\geq\tau l.
\end{equation}

Theorem \ref{theo:conj_exp_cone} below establishes some easily verifiable and satisfiable conditions under which a torus map is conjugate to $F_0$ in (\ref{eqn:map}). We will assume the following on $F$.
\\(A2) there is an invariant, expanding cone-structure centered around the vectors \{$e_1\ldots,e_k$\}.
\\(A3) $DF$ is invertible. Such a map is called \textbf{local diffeomorphism}, i.e., every point has a neighborhood in which the map is a diffeomorphism. 
\\(A4) there is a dominated, invariant, expanding cone-structure centered around the vectors \{$e_1\ldots,e_k$\}.

\begin{theorem}\label{theo:conj_exp_cone}
Let $F:\TorusD\to\TorusD$ be a $C^1$ map that satisfies assumptions (A1)-(A3). Then if $A$ is an expanding matrix, then $F$ is conjugate to a skew-product map of the form (\ref{eqn:map}). Moreover, if (A4) is satisfied, then the conjugacy map $H$ is differentiable along the Y-direction.
\end{theorem}

\begin{corollary}\label{corr:conj_exp_cone} 
Let $M$ be an integer matrix with an integer eigenvalue $m$ with $|m|>1$. Then there is a constant $\delta=\delta(M)>0$ such that if $G:\EucD\to\EucD$ is a $C^1$, $1$-periodic map satisfying $\|G\|_{C^1}<\delta$, then the torus map given by $F(z)=Mz+G(z)\bmod 1$ is conjugate to a map of the form (\ref{eqn:map}). 
\end{corollary}
\begin{proof}
The proof of this corollary starts with two observations,
\\(i)the existence of an expanding cone structure is an open condition in the $C^1$-topology of maps. 
\\(ii) the winding matrix of a torus map is independent of the periodic part $G$ of the map. [See Eq. \ref{eqn:periodic_part}].
Secondly, for the linear torus map given by the winding matrix, namely, $M:\Torus\to\Torus$, there is an expanding cone structure with $K=m$, $\alpha=\infty$, $e^1=v_m$, $e^2=v_1$. Therefore, there is a bound $\delta>0$ such that if $\|G(y)\|_{C^1}<\delta$, then the map still retains the same winding matrix and an expanding cone structure. \qed
\end{proof}

\textbf{Remark.} In particular, if $M$ is symmetric in the above Theorem, then setting $\delta=0.5(|m|-1)$ suffices.

\textbf{Conjugacy results on the torus.} Notice that $F$ in Theorem \ref{theo:conj_exp_cone} is not invertible but an $m$-fold covering map. Our conjugacy theorem therefore involves some conditions different from conjugacy theorems stated about invertible maps in the past. In \cite{EntropyInvar}, the authors prove that two ergodic homeomorphisms of the torus are conjugate via a measurable conjugacy iff they have the same metric entropy. The conclusions of our theorem is similar to a theorem by Ilyashenko and Negut in \cite{SkewProd3}, where they prove a structural stability theorem for invertible step-skew products. Since we consider maps which are not invertible, we have to rely on a different approach to prove conjugacy. Various other numerical invariants of topological conjugacies between linear endomorphisms of the torus have been described in \cite{Conjug3}. They are a ``complete'' system of invariants, i.e., their invariance is necessary and sufficient for topological conjugacy.

%...............................................................................................................................................................................................................................................................
\subsection{Proof of Theorem \ref{theo:conj_exp_cone}}\label{sect:conjg_proof}

The conjugacy can be constructed using the map $\lowPhi$ from  Theorem \ref{thm:factor_Phi_matrix}. We begin with a Proposition.
\begin{proposition}\label{prop:phi_fiber_C0}
The fibers of $\lowPhi$, which are the sets $\lowPhi^{-1}(\theta_0)$ for some $\theta_0\in S^1$, are topologically $\TorusP{d-1}$.
\end{proposition}
\begin{proof}
To prove this, we will first prove that for every $x_0\in\mathbb{R}$, $\upPhi^{-1}(x_0)$ is an $(d-k)$-hypersurface. We will then use the fact that $\lowPhi$ is a factor of $\upPhi$ to prove the claim of the theorem. %The direction $e$ on $\TorusD$ lifts uniquely under the $\bmod 1$ map $:\EucD\to\TorusD$ to a direction which will also be denoted as $e$. $e$ partitions $\EucD$ into a collection of parallel lines. 

\textbf{Claim D.} Two distinct points in $\upPhi^{-1}(x_0)$ cannot be connected by a conic curve. 

To see this, first assume the contrary. So there are distinct points $z_1, z_2\in\upPhi^{-1}(x_0)$ and $\gamma$ is a conic curve joining $z_1$ and $z_2$. Then for every $n\in\mathbb{N}$, $F^n(\gamma)$ is again a conic curve whose endpoints are $F^n(z_1)$ and $F^n(z_2)$, both lying in the fiber $\upPhi^{-1}(A^n x_0)$ by Eq. \ref{eqn:factoring_phi}. By Ineq. \ref{eqn:ReverseLip_Phi}, we can conclude that $|\proj_W(F^n z_1-F^n z_2)|\leq\frac{2}{\|A\|-1}\|G\|_0$. Let $l$ be the length of $\gamma$. Then the length of $F^n(\gamma)$ is at least $K^n\gamma$. By Ineq. \ref{eqn:length_proj}
\[\frac{2}{\|A\|-1}\|G\|_0 \geq \phi(F^n z_1-F^n z_2)| \geq lk^n\tau.\]
This inequality holds for every integer $n$. But while the left hand side remains bounded, the right hand side diverges to $\infty$ as $n\rightarrow\infty$. This leads to a contradiction, so our assumption of the contrary was false. So Claim D must be true.

\textbf{Claim E.} $\upPhi^{-1}(x_0)$ is an embedded $(d-k)$-dimensional hyper-plane. 

To see this, first note that any straight line parallel to $W$ is a conic curve. Therefore $\upPhi^{-1}(x_0)$ intersect every $k$-hyperplane parallel to $W$, by Theorem \ref{thm:factor_Phi_matrix}. This combined with the above claim implies that $\upPhi^{-1}(x_0)$ intersect every $k$-hyperplane parallel to $W$ at a unique point. Since $\upPhi$ is continuous, $\upPhi^{-1}(x_0)$ is a closed set. The set of $k$-hyperplane parallel to $W$ can be parameterized by $\EucP{d-k}$, therefore $\upPhi^{-1}(x_0)$ is the graph of a continuous map $\EucP{d-k}\mapsto \EucD$. So Claim E is true.

We can now show that $\lowPhi^{-1}(\theta_0)$ is a topological $(d-k)$ torus. Let $x_0\in\mathbb{R}$ be some lift of $\theta_0$ under the projection map $proj$. Then $\lowPhi^{-1}(\theta_0)$ is the image under $proj$ of the sets $\upPhi^{-1}(x_0+\vec n)$, where $\vec n\in\mathbb{Z}^k$ ranges over all integers. Note that by Eq. \ref{eqn:up_phi_periodic}, these hyper-surfaces are translates of each other by integer vectors. Therefore, the images under $proj$ of all the hyper-surfaces $\upPhi^{-1}(x_0+\vec n)$ is a single $(d-k)$ torus. This concludes the proof of the proposition.\qed
\end{proof}

\begin{lemma}\label{subsect:phi_as_conj}
Let $H:\EucD\to\EucD$ be defined as $H(z)=(\lowPhi(z),\proj_{W^\perp} z)$, where $\proj_{W^\perp} z$ is the projection onto the last $d-k$ coordinates. Then $H$ is a homeomorphism and $H\circ F\circ H^{-1}$ is of the form given in Eqn. \ref{eqn:map}.
\end{lemma}
\begin{proof}
Since $\lowPhi$ is continuous, $H$ is continuous. We will first prove that $H$ is invertible and then show that, in fact it is a homeomorphism. Finally, we will show that $H$ gives the desired conjugacy. Since $\TorusD$ is a compact set and $H$ is continuous, to prove invertibility, it is enough to show that the map is both one-to-one and onto.

%We begin with the observation that $\phi(z):=\vp\cdot z$ attains a value of $0$ on vectors on the line \{$tv_1=t(u_2,-u_1)$ : $t\in\mathbb{R}$\} and a value of $1$ on the line \{$(n_1,n_2)+t(u_2,-u_1)$ : $t\in\mathbb{R}$\}. These two lines are parallel to the vector $v_1$ and form two opposite edge of $P$. $\phi$ is constant along every line parallel to $v_1$. Also, for every $x_0\in\mathbb{R}$, the interior of the region $P$ intersects exactly one of the curves $\upPhi^{-1}(x_0+n)$ for $n\in\mathbb{N}$.

\textbf{Onto: } Let $z_0=(x_0,y_0)\in\TorusD$, where $x_0\in \TorusP{k}, y_0\in \TorusP{d-k}$. Since the set $R:=$ \{$x\in\TorusD$ : $\proj_{W^\perp} z=y_0$\} is a $k$-torus parallel to $W$, by Theorem \ref{thm:factor_Phi_matrix}, $\lowPhi^{-1}(x_0)$ is a topological $(d-k)$-torus transverse to $R$ and therefore they intersect at a unique point $z$. Therefore, $H(z)=(x_0,y_0)$. 

\textbf{One-to-one: } Consider any two inverse images $z'=(x',y')$ and $z''=(x'',y'')$ of $(x_0,y_0)$. since $\proj_{W^\perp}(z'-z'')=0\bmod 1$, $z'-z''$ is parallel to $W$. However both $z'$ and $z''$ lie on $\lowPhi^{-1}(x_0)$, which is uniformly transverse to all lines parallel to $W$. This forces $z'=z''$. 

\textbf{Conjugacy: } For any $(x_0,y_0)\in\TorusD$, let $(x_1,y_1):=H^{-1}(x,y)$, $(x_2,y_2):=F(x_1,y_1)$ and $(x_3,y_3):=H(x_2,y_2)$. To show that $H$ is the desired conjugacy have to show that $x_3=mx_0\ (\mod~1)$. Note that $x_3=\lowPhi(x_2,y_2)=\lowPhi\circ F(x_1,y_1)$. By Eqn. \ref{eqn:factoring_phi}, $x_3=m\times \lowPhi(x_1,y_1)$. But since $(x_1,y_1)=H^{-1}(x_0,y_0)$, $\lowPhi(x_1,y_1)$ must be equal to $x_0$. Therefore, $x_3=mx_0\ (\mod~1)$. \qed
\end{proof}

%...............................................................................................................................................................................................................................................................
\subsection{The fibers of $\upPhi$}\label{subsect:phi_fiber}

We will now prove that if the dominated, invariant expanding cone condition (A4) is satisfied, then $H$ is differentiable along the fibres $\lowPhi^{-1}(\theta)$, for $\theta\in\TorusP{k}$. It is equivalent to prove that the fibers of $\upPhi$, which are the sets $\upPhi^{-1}(x)$ for $x\in \mathbb{R}^{k}$ are differentiable, embedded $(k-d)$-hyperplanes. Before proving that, we will describe a generalized notion of tangent vectors.

Let $\lambda:(0,1)\rightarrow \mathbb{R}^d$ be a continuous curve. Let $t_0\in(0,1)$ and $z_0=\lambda(t_0)$. For every non-zero vector $v\in\mathbb{R}^d$, let $\hat{v}$ denote the normalized vector $\frac{v}{\|v\|}$, where $\|v\|$ is the Euclidean norm of $v$. A unit vector $\hat v_0$ will be called a \textbf{generalized tangent direction} to $\lambda$ at $z_0$ is there is a sequence $(t_n)_{n\in\mathbb{N}}$ such that $t_n\rightarrow t_0$, and if $v_n$ denotes the vector $\lambda(t_n)-\lambda(t_0)$, then $\hat{v}_n\rightarrow \hat{v}_0$. 

\textbf{Properties of generalized tangent directions}. The following properties of generalized tangent directions follow immediately from their definition.
\begin{enumerate}
\item Since the definition of a generalized tangent direction is a local property, the definition can be extended to continuous curves in manifolds, like $\TorusD$.
\item Every curve has at least one generalized tangent direction at each of its points. This is because, the vectors $\hat{u}_n$ all lie in the unit sphere $S^{d-1}$ of the tangent space at $z_0$. Since this unit circle is compact, for any sequence $t_n\rightarrow t_0$ and $u_n:=\lambda(t_n)-\lambda(t_0)$, the vectors $\hat{u}_n$ will have at least one limit point.
\item Note that an embedded $(d-k)$-manifold is $C^1$ iff there are exactly $(d-k)$ linearly independent generalized tangent directions at each of its points. 
\end{enumerate}
We will need the following lemma to prove our claim.

\begin{lemma}\label{lem:inv_gen_tngnt_vctr}
Let $\lambda:S^1\rightarrow M$ be a continuous curve in a $d$ dimensional manifold $M$. Let $F:M\rightarrow M$ be a local diffeomorphism. Then $DF$ maps each generalized tangent direction of $\lambda$ into a generalized tangent direction of $F(\lambda)$.
\end{lemma}
\begin{proof}
Let $\hat v$ be a generalized tangent direction to $\lambda$ at a point $z_0=\lambda(t_0)$ for some $t_0\in S^1$. We will prove that $DF(z_0)(v)$ is along a generalized tangent direction to $F(\lambda)$ at $F(z_0)$. By definition, there is a sequence $t_n\rightarrow t_0$ such that if $v_n:=\lambda(t_n)-\lambda(t_0)$, then $\hat{v}_n\rightarrow \hat{v}_0$. Let $z_n$ denote the point $\lambda(t_n)$. So $z_n=z_0+v_n$.

Let $\lambda(r_0)$ correspond to the point $F(z_0)$, and similarly, $\lambda(r_n)=F(\lambda(t_n))$, where $r_0,r_1,r_2,\ldots\in S^1$. Then note that $r_n\rightarrow r_0$. By the definition of the derivative of a function $F$, $\underset{n\rightarrow\infty}{\lim}\frac{\|F(z_0+v_n)-(F(z_0)+Df(z_0)v_n)\|}{\|v_n\|}=0$.
\\Therefore, $F(z_n)=F(z_0+u_n)\rightarrow F(z_0)+DF(z_0)v_n$ or $F(z_n)-F(z_0)\rightarrow DF(z_0)v_n$.
\\Therefore, $\frac{F(z_n)-F(z_0)}{\|F(z_n)-F(z_0)\|}\rightarrow\frac{DF(z_0)v_n}{\|DF(z_0)v_n\|}$.
\\But since $v_n\rightarrow v_0$, we must have that $DF(z_0)v_n\rightarrow DF(z_0)v_0$ or $\frac{DF(z_0)v_n}{\|DF(z_0)v_n\|}\rightarrow\frac{DF(z_0)v_0}{\|DF(z_0)v_0\|}$.
\\Therefore, $\frac{F(z_n)-F(z_0)}{\|F(z_n)-F(z_0)\|}\rightarrow\frac{DF(z_0)v_0}{\|DF(z_0)v_0\|}$.
\\Therefore, $DF(z_0)v_0$ must be a generalized tangent direction to the curve $F(\lambda)$ at the point $F(z_0)$. \qed
\end{proof}

The proof will be by contradiction. So there is some point $z\in \upPhi^{-1}(x_0)$ with $d-k+1$ linearly independent generalized tangent directions \{$v_0,\ldots,v_{d-k}$\}. At least one of these vectors is not orthogonal to $e(z)$, say $v_0$. By Eqn. \ref{eqn:factoring_phi} and Lemma \ref{lem:inv_gen_tngnt_vctr}, \{$DF^n v_0,\ldots,DF^n v_{d-k}$\} are mapped into generalized tangent directions on $\upPhi^{-1}(m^n x_0)$, since $DF$ is a local diffeomorphism, these directions are independent. Because (A4) is satisfied, if $n$ is large enough, then $DF^n v_0\in e(F^n z)$, the cone at $z$. However, it follows from Claim D in the proof of Proposition \ref{prop:phi_fiber_C0} that a fiber of $\upPhi$ cannot have a generalized tangent direction lying inside cone around \{$e_1,\ldots,e_k$\}. This leads to a contradiction and completes the proof. \qed

%-_-_-_-_-_-_-_-_-_-_-_-_-_-_-_-_-_-_-_-_-_-_-_-_-_-_-_-_-_-_-_-_-_-_-_-_-_-_-_-_-_-_-_-_-_-_-_-_-_-_-_
\bibliographystyle{unsrt}
\bibliography{Conjugacy_bibliography}

\begin{thebibliography}{10}

\bibitem{Conjug2}
R~L Adler and R~Palais.
\newblock Homeomorphic conjugacy of automorphisms on the torus.
\newblock {\em Proc. Amer. Math. Soc.}, 16(6):1222--1225, 1965.

\bibitem{ExpndngEndo}
M~Shub.
\newblock Topological conjugacy of linear endomorphisms of the 2-torus.
\newblock {\em Ame. J. Math.}, 91(1):175--199, 1969.

\bibitem{Lehmer_polyn}
D~H Lehmer.
\newblock Factorization of certain cyclotomic functions.
\newblock {\em Nonlinearity}, 34(3):461--479, 1933.

\bibitem{SkewProd5}
A~Berger and S~Siegmund.
\newblock On the gap between random dynamical systems and continuous skew
  products.
\newblock {\em Journal of Dynamics and Differential Equations}, 15 (2-3), 2003.

\bibitem{SkewProd2}
V~Kleptsyn and M~B Nalskii.
\newblock Contraction of orbits in random dynamical systems on the circle.
\newblock {\em Functional Analysis and Its Applications}, 38 (4):267–282,
  2004.

\bibitem{SkewProd1}
A~J Homburg.
\newblock Circle diffeomorphisms forced by expanding circle maps.
\newblock {\em Ergod. Th. and Dynam. Sys.}, 32:2011–2024, 2012.

\bibitem{Kostelich}
E~J Kostelich, I~Kan, C~Grebogi, E~Ott, and J~A Yorke.
\newblock Unstable dimension variability: A source of nonhyperbolicity in
  chaotic systems.
\newblock {\em Physica D}, 109, (1-2):81–90, 1 November 1997.

\bibitem{InvsblAttrctr}
Y~Ilyashenko and A~Negut.
\newblock Invisible parts of attractors.
\newblock {\em Nonlinearity}, 23(5):1199, 2010.

\bibitem{SkewProd4}
Victor Kleptsyn and Denis Volk.
\newblock Nonwandering sets of interval skew products.
\newblock {\em Nonlinearity}, 27 (7), 2014.

\bibitem{SkewProd3}
Y~Ilyashenko and A~Negut.
\newblock Holder properties of perturbed skew products and fubini.
\newblock {\em Nonlinearity}, 25 (8):2377--2411, 2012.

\bibitem{Rudin1}
Walter Rudin.
\newblock {\em Principles of mathematical analysis}.
\newblock McGraw-Hill Science/Engineering/Math., 1976.

\bibitem{SeminalDominated}
Richard Mañé.
\newblock A proof of the $c^1$-stability conjecture.
\newblock {\em Publications Mathématiques de l'Institut des Hautes Études
  Scientifiques}, 66, (1):161--210, 1987.

\bibitem{DomSplit}
E~Pujals and M~Sambarino.
\newblock On the dynamics of dominated splitting.
\newblock {\em Annals of Mathematics}, 169:675--740, 2009.

\bibitem{EntropyInvar}
R~Adler and B~Weiss.
\newblock Entropy a complete metric invariant for automorphisms of the torus.
\newblock {\em Proceedings of the National Academy of Sciences of the United
  States of America PNAS}, 57 (6):1573--1576, June 1, 1967.

\bibitem{Conjug3}
R~Adler, C~Tresser, and P~Worfolk.
\newblock Topological conjugacy of linear endomorphisms of the 2-torus.
\newblock {\em Trans. Amer. Math. Soc.}, 349(4):1633--1652, 1997.

\end{thebibliography}
\end{document}